\documentclass[a4paper,11pt,reqno]{amsart}

\usepackage{fullpage}
\usepackage{amsmath,amsfonts,amssymb,amsthm}
\usepackage{url}
\usepackage[normalem]{ulem}
\usepackage{mathtools}
\usepackage{enumitem} 
\usepackage{color}
\usepackage{caption}
\usepackage[vcentermath]{youngtab}
\usepackage{tikz}
\usepackage{hyperref} 
\usetikzlibrary{decorations.pathreplacing,matrix,arrows,decorations.pathmorphing}
\usepackage{mathrsfs}

\DeclareMathOperator{\Hom}{{\rm Hom}}
\DeclareMathOperator{\sgn}{{\rm sgn}}

\DeclareMathOperator{\id}{\rm id}
\DeclareMathOperator{\Inf}{\rm Inf}

\renewcommand{\bar}{\overline}
\renewcommand{\tilde}{\widetilde}
\newcommand{\ind}{\uparrow}

\newcommand{\ubar}{\underline}
\newcommand{\tsp}[1]{\textsuperscript{#1}}

\newcommand{\lam}{\lambda}
\newcommand{\roweq}{\sim_{\text{row}}}

\newcommand{\dc}{10}

\theoremstyle{plain}
\newtheorem{thm}{Theorem}[section]

\newtheorem{lem}[thm]{Lemma}
\newtheorem{conj}[thm]{Conjecture}

\newtheorem{repthm}{Theorem}[thm]

\theoremstyle{definition}

\newtheorem{ex}[thm]{Example}

\theoremstyle{remark}
\newtheorem*{rem}{Remark}

\newcommand{\rtbltwo}[4]{\begin{tikzpicture}[baseline={([yshift=-.8ex]current bounding box.center)}]
\draw (0,0)--(0.8,0);
\draw (0,-0.4)--(0.8,-0.4);
\draw (0,-0.8)--(0.8,-0.8);
\node (a1) at (0.2,-0.2){#1};
\node (a2) at (0.6,-0.2){#2};
\node (a3) at (0.2,-0.6){#3};
\node (a4) at (0.6,-0.6){#4};
\end{tikzpicture}}

\begin{document}

\title{\bf Relationships between plethysm coefficients}
\author{Melanie de Boeck}
\date{\today}
\address{School of Mathematics, Statistics \& Actuarial Science, University of Kent, Canterbury, Kent CT2 7NF, UK}
\email{m.deboeck.1@bham.ac.uk}
\curraddr{School of Mathematics, University of Birmingham, Edgbaston, Birmingham, B15 2TT, UK}

\begin{abstract}
We consider the plethysm problem stated for representations of symmetric groups. In particular, we prove new relationships between composition multiplicities of twisted Foulkes modules. Expressed in terms of symmetric functions, our results establish relationships between plethysm coefficients.
\end{abstract}

\maketitle

\thispagestyle{empty}
\section{Introduction}\label{intro}
The plethysm problem is a fundamental question that may be stated for representations of symmetric groups, representations of general linear groups, or symmetric functions.  Our work takes the foremost approach, but we begin by describing the symmetric function setting so that we may outline the historical development of the problem.

Plethysm multiplication of symmetric functions was introduced by Littlewood in~\cite{littlewood} in 1936. Given two partitions $\nu$ and $\mu$ of $n,m\in\mathbb{Z}$, respectively, the plethysm $s_\nu\circ s_\mu$ of the Schur functions $s_\nu$, $s_\mu$ may be expressed as a linear combination of Schur functions with integer coefficients: 
\[s_\nu\circ s_\mu=\sum_{\lam\vdash mn}p^\lam_{\nu,\mu}s_\lam.\]

A complete understanding of the \emph{plethysm coefficients} $p^\lam_{\nu,\mu}$ remains elusive: in his millennium survey~\cite[Problem~9]{stanley}, Stanley lists the task of obtaining a combinatorial description of these coefficients as one of the key open problems in algebraic combinatorics. In the language of representations of symmetric groups, the plethysm problem asks for the composition multiplicities of certain modules for $\mathfrak{S}_{mn}$, which are modules that are induced from wreath products $\mathfrak{S}_m\wr\mathfrak{S}_n$; we shall make this precise later in Equation~\eqref{eq:symm_pleth_problem}.\medskip

The two main results of this paper are new relationships between plethysm coefficients in the case where $\mu=(1^m)$. We use the following notation in the statement of the theorems: given any two partitions $\lam$ and $\theta$, define $\lam+\theta:=(\lam_1+\theta_1,\lam_2+\theta_2,\ldots)$. 

\begin{thm}If \label{pleth_thm1} $\lam$ is a partition of $mn$, then \begin{equation*}p^\lam_{(n),(1^m)}\leq p^{\lam+(1^m)}_{(n+1),(1^m)}.\end{equation*}
\end{thm}

\begin{thm}Let \label{pleth_thm2} $\lam$ be a partition of $mn$ with fewer than $2m$ parts. If $m$ is even, then, for any $a\geq 2$, \begin{equation*}p^\lam_{(n),(1^m)}\leq p^{\lam+(1^{am})}_{(n+a),(1^m)}.\end{equation*}
\end{thm}

There exists an involution $\omega$ on the ring of symmetric functions (see~\cite[Ch.~I, \S2, Equation~(2.7)]{macdonald}), using which it is possible to obtain further results about plethysm coefficients. In the language of symmetric groups, applying $\omega$ simply corresponds to tensoring with the sign representation. A consequence for the plethysm coefficients is as follows:
\[p^\lam_{\nu,\mu}=\begin{dcases*}p^{\lam'}_{\nu,\mu'}& if $|\mu|$ is even;\\ p^{\lam'}_{\nu',\mu'}& if $|\mu|$ is odd,\end{dcases*}\] where $\lam'$, say, denotes the conjugate partition of $\lam$.\smallskip

Applying $\omega$ to Theorems~\ref{pleth_thm1} and~\ref{pleth_thm2}, we obtain the following equivalent relationships.

\setcounter{thm}{1}
\begin{repthm}Let $\lam=(\lam_1,\lam_2,\ldots,\lam_\ell)$ be a partition of $mn$ and let $q\in\mathbb{N}_0$ be minimal such that $0\leq\lam_{q+1}\leq m$. \label{row_m_thm} Define $\tilde\lam:=(\lam_1,\ldots,\lam_q,m,\lam_{q+1},\ldots,\lam_\ell)$ if $q\neq0$ and $\tilde\lam:=(m,\lam_1,\ldots,\lam_\ell)$ if $q=0$.
\begin{enumerate}
\item If $m$ is even, then $p^\lam_{(n),(m)}\leq p^{\tilde\lam}_{(n+1),(m)}$.
\item If $m$ is odd, then $p^\lam_{(1^n),(m)}\leq p^{\tilde\lam}_{(1^{n+1}),(m)}$.
\end{enumerate}
\end{repthm}

\setcounter{thm}{2}\setcounter{repthm}{0}
\begin{repthm}Let $\lam=(\lam_1,\ldots,\lam_\ell)$ be a partition of $mn$ such that $0\leq\lam_1<2m$ \label{row_am_thm} and for any $a\geq 2$, define $\tilde\lam:=(am,\lam_1,\ldots,\lam_\ell)$. If $m$ is even, then $p^\lam_{(n),(m)}\leq p^{\tilde\lam}_{(n+a),(m)}$.
\end{repthm}

Whilst there exist other relationships between plethysm coefficients (see \S\ref{existing}), our results are novel since they are the first known relationships for which $m$ is fixed. We now reformulate the plethysm problem into the setting that we will use to prove our results. In the sequel, all modules under consideration are right modules. We will work over $\mathbb{C}$ throughout, but the ground field may be replaced by any field of characteristic zero.

Given any $\mathbb{C}\mathfrak{S}_n$-module $X$, we may construct a $\mathbb{C}(\mathfrak{S}_m\wr\mathfrak{S}_n)$-module $\Inf X$ by inflating along the canonical surjection $\mathfrak{S}_m\wr\mathfrak{S}_n\to\mathfrak{S}_n$. When $X$ is a Specht module $S^\nu$ labelled by a partition $\nu$ of $n$, inducing $\Inf S^\nu$ to $\mathfrak{S}_{mn}$ results in a  $\mathbb{C}\mathfrak{S}_{mn}$-module $H_\nu^{(m^n)}=\big(\Inf S^\nu\big)\big\ind_{\mathfrak{S}_m\wr\mathfrak{S}_n}^{\mathfrak{S}_{mn}}$, which we call a \emph{twisted Foulkes module}.  Under the characteristic map, the image of the ordinary character afforded by $H_\nu^{(m^n)}$ is the plethysm $s_\nu\circ s_{(m)}$ (see~\cite[Ch.~I, Appendix~A, Equation~(6.2)]{macdonald}). Moreover, asking for the decomposition of $H^{(m^n)}_\nu$ as a direct sum of irreducible Specht modules is entirely equivalent to the problem of determining plethysm coefficients for the corresponding plethysm. In particular, 
\begin{equation}H^{(m^n)}_\nu=\bigoplus_{\lam\vdash mn}p^\lam_{\nu,(m)}S^\lam.\label{eq:symm_pleth_problem}\end{equation}

The approach that we will take to prove Theorems~\ref{row_m_thm} and~\ref{row_am_thm} is to prove the existence of appropriate homomorphisms $S^\lam\to H_\nu^{(m^n)}$ and then use these homomorphisms to establish a lower bound on $p^\lam_{\nu,(m)}$. In the sequel, we will refer to $H^{(m^n)}:=H^{(m^n)}_{(n)}$ as a \emph{Foulkes module} and $K^{(m^n)}:=H^{(m^n)}_{(1^n)}$ as a \emph{signed Foulkes module}.\medskip

After a survey of existing results about plethysm coefficients in Section~\ref{existing}, we will recall the theory of semistandard homomorphisms in Section~\ref{rep_theory} and indicate how it may be used to study signed Foulkes modules. The techniques that we use are of independent interest: they have never before been employed as a way to study signed Foulkes modules and they provide a practical approach to studying twisted Foulkes modules in general.  Section~\ref{prelim} will be devoted to preliminaries and to defining the tableaux upon which the proofs of our main theorems will rely. In Sections~\ref{row_m_thm_proof} and~\ref{row_am_thm_proof} we will prove Theorems~\ref{row_m_thm} and~\ref{row_am_thm}, respectively, and we conclude with further conjectures about relationships between plethysm coefficients in Section~\ref{conjectures}. 

\section{Background on plethysm coefficients}\label{existing}
As might be expected, many of the existing results on plethysm coefficients have been proved via symmetric functions.  However, valuable contributions have also been made using representation theoretic techniques. The plethysm $s_\nu\circ s_\mu$ has been fully described for particular choices of $\nu$ and $\mu$: $s_{(n)}\circ s_{(2)}$, $s_{(2)}\circ s_{(m)}$ and $s_{(3)}\circ s_{(m)}$ (work by Thrall~\cite{thrall}); and $s_{(1^n)}\circ s_{(2)}$ and $s_{(1^2)}\circ s_{(m)}$ (see~\cite[Ch.~I, \S8]{macdonald}). Additionally, several algorithms have been posed for calculating $p^\lam_{\nu,\mu}$ when $\nu$ is any partition of two, three or four (see \cite{duncan,foulkes} and~\cite[\S3.5]{howe}). There are also results about plethysm coefficients when the partition $\lam$ takes a specific form. For example, minimal and maximal partitions that label Schur functions appearing in the plethysms $s_\nu\circ s_{(m)}$ are determined in~\cite{paget_wildon_twisted}, and when $\lam$ is an even partition or a `hook-like' partition, the coefficients $p^\lam_{\nu,\mu}$ are studied in~\cite{bci,mm} and~\cite{giannelli,langley}, respectively.

We now collect together a few results that can be found scattered throughout the literature. Like the main theorems in this paper, they concern relationships between plethysm coefficients.

\begin{enumerate}[label={\bf\ref*{existing}.\arabic*},ref={\ref*{existing}.\arabic*}]
\item Ikenmeyer~\cite[Proposition~4.3.4]{ikenmeyer}: Let $\lam$ be a partition of $mn$ and let $\theta$ be a partition of $kn$. If $p^\lam_{(n),(m)}>0$ and $p^\theta_{(n),(k)}>0$, then $p^{\lam+\theta}_{(n),(m+k)}>0$.
\item Foulkes' Second Conjecture, proved by Brion in~\cite{brion}: If $\lam$ is a partition of $mn$, then $p^\lam_{(n),(m)}\leq p^{\lam+(n)}_{(n),(m+1)}$.\label{FSC}
\item Newell~\cite{newell}: If $\lam\vdash mn$, then $p^\lam_{(n),(m)}=p^{\lam+(1^n)}_{(1^n),(m+1)}$ and $p^\lam_{(1^n),(m)}=p^{\lam+(1^n)}_{(n),
(m+1)}$. \label{Newell_statement}
\item Foulkes' Conjecture~\cite{foulkes}, most recently proved to hold for $m=5$ in~\cite{cim}: If $m<n$, then $p^\lam_{(n),(m)}\geq p^\lam_{(m),(n)}$ for all partitions $\lam$ of $mn$.
\end{enumerate} 

Whilst~\ref{FSC} and \ref{Newell_statement} were both first proved in the symmetric function setting, alternative proofs can be given in the language of representations of symmetric groups. More significantly, using the latter setting Dent~\cite[Theorem~3.10]{dent} obtains additional information: the plethysm coefficients that are the subject of~\ref{FSC} are in fact equal whenever $\lambda_2\leq m$.

\section{Background material from the representation theory of symmetric groups}\label{rep_theory}

In this section, we remain consistent with the notation given by James in~\cite[\S13]{james}. For a partition $\mu$ of $n$, we let $M^\mu$ denote the Young permutation module, that is the $\mathbb{C}\mathfrak{S}_n$-module spanned by all $\mu$-tabloids $\{t\}$. Recall that we represent a $\mu$-tabloid $\{t\}$ by only drawing lines between the rows of the representative $t$. For a $\mu$-tableau $t$, we denote the column stabiliser of $t$ by $C_t$ and we set $\kappa_t=\sum_{\pi\in C_t} \sgn(\pi)\pi$. The Specht module $S^\mu$ is the submodule of $M^\mu$ that is spanned by all $\mu$-polytabloids $e_t=\{t\}\kappa_t$. Moreover, $S^\mu$ is a cyclic module, generated by any single polytabloid. The Specht modules are particularly important, since $\{S^\mu\mid\mu\vdash n\}$ is a complete set of non-isomorphic, irreducible $\mathbb{C}\mathfrak{S}_n$-modules,~\cite[Theorem~4.12]{james}.

\subsection{Semistandard homomorphisms}\label{sstd_homs}
We now review a second, equivalent description of Young permutation modules, which requires a kind of tableau that is permitted to have repeated entries; to avoid any confusion, we will use capital letters to denote such tableaux. Let $\lambda$ be a partition of $n$ and let $\mu=(\mu_1,\ldots,\mu_k)$ be a composition of $n$. A $\lambda$-tableau $T$ is said to be \emph{of type $\mu$} if for every positive $j\in\mathbb{Z}$, the integer $j$ occurs $\mu_j$ times in $T$; write \[\mathcal{T}(\lambda,\mu):=\{T\mid T\text{ is a }\lambda\text{-tableau of type }\mu\}.\] A tableau $T\in\mathcal{T}(\lambda,\mu)$ is called \emph{semistandard} if the numbers are non-decreasing along rows of $T$ and strictly increasing down the columns of $T$. We write $\mathcal{T}_0(\lambda,\mu)$ to denote the set of semistandard tableaux in $\mathcal{T}(\lambda,\mu)$.

Henceforth, fix a $\lambda$-tableau, $t$. There is a well-defined action of $\mathfrak{S}_n$ on $\lam$-tableaux of type $\mu$.  Take $T\in\mathcal{T}(\lambda,\mu)$. If we let $(i)T$ be the entry in $T$ which occurs in the same position as $i$ occurs in $t$, then we may define the action of $\mathfrak{S}_n$ on $\mathcal{T}(\lambda,\mu)$ by \begin{equation*}(i)(T\pi)=\big(i\pi^{-1}\big)T,\end{equation*} where $T\in \mathcal{T}(\lambda,\mu)$, $\pi\in\mathfrak{S}_n$ and $1\leq i\leq n$. With this action, take $M^\mu$ to be the $\mathbb{C}\mathfrak{S}_n$-module spanned, as a vector space, by $\lambda$-tableaux of type $\mu$. The equivalence of the two definitions of $M^\mu$ is established by the following isomorphism: given a fixed $\lambda$-tableau $t$, the image of $T\in\mathcal{T}(\lambda,\mu)$ under the isomorphism is the $\mu$-tabloid $\{t_T\}$, which is obtained by putting $i$ in row $j$ of $\{t_T\}$ if $(i)T=j$. 

The tableaux $T_1,\,T_2\in\mathcal{T}(\lam,\mu)$ are said to be \emph{row equivalent} if $T_2=T_1\pi$ for some $\pi\in R_t$, the row stabiliser of $t$; write $T_1\roweq T_2$. With this in mind, if $T\in \mathcal{T}(\lam,\mu)$, then there is a well-defined map $\hat\theta_T: M^\lambda\to M^\mu$ (defined by James in~\cite[Definition~13.3]{james}), defined on $\lam$-tabloids by \begin{equation*}\hat\theta_T: \{t\}\longmapsto\sum_{T'\sim_{\rm row} T}T',\end{equation*} which can be extended to a homomorphism by allowing group elements to act. Let $\theta_T\in\Hom_{\mathbb{C}\mathfrak{S}_n}\left(S^\lam,M^\mu\right)$ be the restriction of $\hat\theta_T$ to the Specht module $S^\lam$. If $T$ is a semistandard tableau, then we call $\theta_T$ a \emph{semistandard homomorphism}. 

In~\cite[13.5]{james}, James states that, given $\lam,\,\mu\vdash n$, a column of $T\in\mathcal{T}(\lam,\mu)$ contains two identical numbers if and only if $T\kappa_t=0$. So, since 
\begin{equation*}(e_t)\theta_T=\left(\{t\}\kappa_t \right)\theta_T=(\{t\})\theta_T\kappa_t=\left(\sum\nolimits_{T'\roweq T}T'\right)\kappa_t=\sum\nolimits_{T'\roweq T}(T'\kappa_t),
\end{equation*}
it is clear that sometimes $\theta_T$ is the zero map. However, by restricting our attention to semistandard tableaux, we are able to guarantee that the corresponding semistandard homomorphisms are non-zero. A basis for $\Hom_{\mathbb{C}\mathfrak{S}_n}\left(S^\lambda,M^\mu\right)$ is given by $\left\{\theta_T\,\,\big\vert\,\,T\in \mathcal{T}_0(\lambda,\mu)\right\}$; this is~\cite[Theorem~13.13]{james}.

\begin{ex}\label{sstd_hom_ex} Take $\lambda=(3,1)$ and $\mu=(2^2)$. If $t=\young(123,4)$ and $T=\young(112,2)$, then \linebreak$\kappa_t=1-(1\,\,4)$ and \allowdisplaybreaks\begin{eqnarray*} \theta_T:e_t &\longmapsto&\left(\,\young(112,2)+\young(121,2)+\young(211,2)\,\right)\kappa_t\\
&=&\young(112,2)+\young(121,2)-\young(212,1)-\young(221,1)\\
&=&\rtbltwo{$1$}{$2$}{$3$}{$4$}+\rtbltwo{$1$}{$3$}{$2$}{$4$}-\rtbltwo{$2$}{$4$}{$1$}{$3$}-\rtbltwo{$3$}{$4$}{$1$}{$2$}\;.
\end{eqnarray*}\end{ex}

The case $\mu=(m^n)$ will be particularly important to us. In this case, we define an action $T*\sigma$ of $\sigma\in\mathfrak{S}_n$ on $T\in \mathcal{T}\big(\lambda,(m^n)\big)$ by $j\mapsto (j)\sigma$, which serves to `relabel' the entries of $T$.

\subsection{The setting for studying Foulkes modules}\label{foulkes_setting}
One way to study the structure of Foulkes modules is to look for maps from the Specht module $S^\lam$ into $H^{(m^n)}$, where $\lam\vdash mn$. In particular, if we find such a non-zero map, then we have identified $S^\lam$ as a composition factor of $H^{\left(m^n\right)}$ and -- equivalently -- verified that $p^\lam_{(n),(m)}\neq 0$. 

Observe that there is a natural surjection $\psi$ from $M^{\left(m^n\right)}$ to $H^{\left(m^n\right)}$, defined on $\left(m^n\right)$-tabloids by mapping $\{t\}$ to the set partition consisting of the $n$ sets whose elements are the entries of the $n$ rows of $\{t\}$.  Therefore, we may construct a map $\bar\theta_T:S^\lam\to H^{\left(m^n\right)}$ by composing a semistandard homomorphism $\theta_T$ with the surjection $\psi:M^{\left(m^n\right)}\to H^{\left(m^n\right)}$. In fact, since $\mathbb{C}\mathfrak{S}_{mn}$ is semisimple and $\{\theta_T\mid T\in\mathcal{T}_0(\lam,\mu)\}$ is a basis of $\Hom_{\mathbb{C}\mathfrak{S}_{mn}}\left(S^\lambda,M^{(m^n)}\right)$, all $\mathbb{C}\mathfrak{S}_{mn}$-homomorphisms from $S^\lambda$ to $H^{\left(m^n\right)}$ will be linear combinations of the maps $\bar\theta_T=\theta_T\circ\psi$.

Computationally, it is often easier to work completely with the description of $M^{(m^n)}$ in terms of $\lam$-tableaux of type $(m^n)$, thus avoiding the need to involve $(m^n)$-tabloids in the calculations. Let $\bar T$ denote the image of the tableau $T\in M^{(m^n)}$ under $\psi$. Observe that the two $\lam$-tableaux of type $(m^n)$, say $T_1$ and $T_2$, are \emph{equivalent under $\psi$} (i.e. $\bar{T_1}=\bar{T_2}$) if there exists a relabelling permutation $\pi\in\mathfrak{S}_n$ such that $T_1*\pi=T_2$. This says that entries in two equivalent tableaux will have the same \emph{pattern}. For example, if \begin{equation*}t=\young(123,4),\quad T_1=\young(112,2)\quad\text{and}\quad T_2=\young(221,1),\end{equation*} so that $T_1*(1\,2)=T_2$, then the $(2^2)$-tabloids corresponding to $T_1$ and $T_2$ are \begin{equation*}\rtbltwo{1}{2}{3}{4}\quad\text{and}\quad\rtbltwo{3}{4}{1}{2},\end{equation*} respectively, and $\bar{T_1}=\bar{T_2}$. We should note that if two tableaux $T_1,T_2\in\mathcal{T}_0\big(\lam,(m^n)\big)$ are equivalent under $\psi$, then this will be true regardless of the choice of $t$ because $t$ just serves as a labelling tableau. 

Unfortunately, it is rarely obvious whether the composition of $\theta_T$ and $\psi$ will yield a non-zero map. Indeed, if we compose the map $\theta_T$ from Example~\ref{sstd_hom_ex} with $\psi$, then all terms in the image of $\bar\theta_T$ will cancel. Since the chosen $T$ is the only semistandard $(3,1)$-tableau of type $(2^2)$, the conclusion is that $S^{(3,1)}$ is not a composition factor of $H^{(2^2)}$. Hence, using this method, it is seldom a trivial task to determine whether or not a Specht module is a composition factor of a Foulkes module.

\subsection{The setting for studying signed Foulkes modules}\label{signed_setting}
The setting for signed Foulkes modules is very similar. Let $\mu$ be a partition of $n$. Recall from Fulton~\cite[\S7.4]{fulton} that $\tilde M^\mu$ is the vector space spanned by all \emph{oriented $\mu$-column tabloids} $|t|$, corresponding to $\mu$-tableaux $t$. This means that if $\sigma\in C_t$, then $|t\sigma|=\sgn(\sigma)|t|$.

Take $\mu=(1^n)$. The signed Foulkes module $K^{(m^n)}$ is the vector space spanned by all oriented $(1^n)$-column tabloids \begin{equation*}|X|=\begin{vmatrix}X_1\\X_2\\\vdots\\X_n\end{vmatrix}\end{equation*} such that the entries of $X$ are disjoint sets $X_i$, each of size $m$, and $\bigcup_{i=1}^n X_i=\{1,2,\ldots,mn\}$. The symmetric group $\mathfrak{S}_{mn}$ acts in the obvious way, permuting $1,2,\ldots,mn$. If a permutation in $\mathfrak{S}_{mn}$ has the effect of swapping exactly two of the sets $X_i$ within the oriented column tabloid, then the resulting element has the opposite orientation and so differs from the original one only by a sign.

Analogous to $\psi:M^{(m^n)}\to H^{(m^n)}$, there is a well-defined, surjective map ${\phi:M^{(m^n)}\to K^{(m^n)}}$, which is defined on $(m^n)$-tabloids by \begin{equation*}
\{t\}=\begin{array}{ccc}\hline
x_1& \ldots &x_m \\\hline
x_{m+1}&\ldots & x_{2m}\\\hline
&\vdots& \\\hline
x_{(n-1)m+1}&\ldots&x_{nm}\\\hline
\end{array}\longmapsto \begin{vmatrix}\{x_1,\ldots,x_m\}\\\{x_{m+1},\ldots,x_{2m}\}\\\vdots\\\{x_{(n-1)m+1},\ldots,x_{nm}\}\end{vmatrix}.
\end{equation*}
Therefore, for any Specht module $S^\lam$, we may construct a homomorphism $\ubar\theta_T: S^\lam \to K^{(m^n)}$ by composing a semistandard homomorphism $\theta_T:S^\lam\to M^{(m^n)}$ with $\phi$, and pose the question: is $\ubar\theta_T$ a non-zero homomorphism? As in the Foulkes module setting, all $\mathbb{C}\mathfrak{S}_{mn}$-homomorphisms $\ubar\theta_T:S^\lam\to K^{(m^n)}$ arise in this way.

It is also appropriate to use the alternative description of $M^{(m^n)}$ in this setting and thus work entirely with $\lam$-tableaux of type $(m^n)$. Let $\ubar T$ denote the image of a tableau $T\in M^{(m^n)}$ under $\phi$. If $t$ is the fixed labelling tableau, then $\ubar T=|X|$, where $|X|$ is the oriented column tabloid whose entries are (in order) the sets $X_i=\{x\mid (x)T=i\}$. For example, if $t$ and $T$ are as in Example~\ref{sstd_hom_ex}, then $(1)T=(2)T=1$ and $(3)T=(4)T=2$ and so\begin{equation*}
\phi:T\longmapsto \begin{vmatrix}\{1,2\}\\ \{3,4\}\end{vmatrix}.
\end{equation*}
Working with this description of $M^{(m^n)}$, we need to pay attention to more than just the pattern of the entries in the tableaux; given $T_1,T_2\in\mathcal{T}\big(\lam,(m^n)\big)$ such that $T_1*\pi=T_2$ for $\pi\in \mathfrak{S}_n$, we must also record the sign of the permutation $\pi$. Indeed, swapping two rows in the $(m^n)$-tabloid yields, under $\phi$, two elements of $K^{(m^n)}$ which differ by a sign. For example, \begin{equation*}\left(\rtbltwo{1}{2}{3}{4}\right)\phi=\begin{vmatrix}
\{1,2\} \\
\{3,4\} \\
\end{vmatrix}=-\begin{vmatrix}
\{3,4\} \\
\{1,2\} \\
\end{vmatrix}=-\left(\rtbltwo{3}{4}{1}{2}\right)\phi.\end{equation*} In general this means that if $T_1*\pi=T_2$, then $\sgn(\pi)(T_1)\phi=(T_2)\phi$. If we reconsider Example~\ref{sstd_hom_ex}, this time composing $\theta_T$ with $\phi$, then we see that \begin{equation*}(e_t)\ubar\theta_T=2\left(\begin{vmatrix}
\{1,2\} \\
\{3,4\} \\
\end{vmatrix}+\begin{vmatrix}
\{1,3\} \\
\{2,4\} \\
\end{vmatrix}\right).\end{equation*}
In other words, we have found a non-zero homomorphism and thus we can conclude that $S^{(3,1)}$ is a summand of $K^{(2^2)}$. In fact, $K^{(2^2)}\cong S^{(3,1)}$.

\section{Preliminaries}\label{prelim}
We begin this section by introducing some definitions and notation concerning tableaux which we will use throughout the remainder of this work. 

If $\lam$ is a partition of $n$ and $t$ is any $\lam$-tableau, then define $t_i^{(j)}$ to be the entry of $t$ in the $i$\tsp{th} row and $j$\tsp{th} column. Further, define \begin{equation*}C_t^{(j)}:=\mathfrak{S}_{\left\{t_1^{(j)},t_2^{(j)},\ldots,t_{\ell_j}^{(j)}\right\}},\end{equation*} where $\ell_j$ denotes the number of entries in column $j$ of $t$, so that we may write the column stabiliser of $t$ as $C_t=C_t^{(1)}\times C_t^{(2)}\times\ldots\times C_t^{(\lam_1)}$, where $\lam_1$ is the first part of $\lam$. Furthermore, $\kappa_t=\prod_{j=1}^{\lam_1}\kappa_t^{(j)}$, where $\kappa_t^{(j)}:=\sum_{\pi\in C_t^{(j)}} \sgn(\pi)\pi$.  Similarly, for $T\in\mathcal{T}(\lam,\mu)$, where $\mu$ is a composition of $n$, let $T_i^{(j)}$ be the entry in the $i$\tsp{th} row of the $j$\tsp{th} column of $T$ and denote the $j$\tsp{th} column of $T$ by $T^{(j)}$.

\subsection{Tableaux for Theorem~\ref{row_m_thm}}\label{m_tableaux}
When we come to prove part 1 of Theorem~\ref{row_m_thm}, we will show (under the assumptions of the theorem) that if $\left(\bar\theta_T:S^{\lam} \rightarrow H^{(m^n)}\right)\neq 0$ for some $\lambda$-tableau $T$ of type $(m^n)$, then $\left(\bar\theta_{\tilde T}:S^{\tilde\lam} \rightarrow H^{(m^{n+1})}\right)\neq 0$, where $\tilde T$ is an appropriately chosen $\tilde\lam$-tableau of type $(m^{n+1})$. The proof of part 2 of Theorem~\ref{row_m_thm} will proceed similarly. For now, we will present a candidate for $\tilde T$.

Suppose that $\lam=(\lam_1,\ldots,\lam_\ell)$ is a partition of $mn$ and let $T$ be a $\lam$-tableau of type $(m^n)$. Define $\tilde T$ in the following way: \begin{equation}\tilde T_i^{(j)}:=\begin{dcases*}T_i^{(j)}&if $i\leq q$;\\ n+1 &if $i=q+1$ and $j\in\{1,2,\ldots,m\}$;\\ T_{i-1}^{(j)} &if $i>q+1$. \end{dcases*}\label{tildeT}\end{equation} 
If $T$ is semistandard, then the construction of $\tilde T$ ensures that $\tilde T$ has distinct entries in columns, and that the entries are non-decreasing along rows.  However, $\tilde T$ is certainly not semistandard in general. Take $t$ to be the $\lam$-tableau which has the digits $1,2,\ldots,mn$ in increasing order along rows. Define the labelling tableau $\tilde t$ by \begin{equation*}\tilde t_i^{(j)}:=\begin{dcases*}t_i^{(j)}&if $i\leq q$;\\ mn+j &if $i=q+1$ and $j\in\{1,2,\ldots,m\}$;\\ t_{i-1}^{(j)}&if $i>q+1$.\end{dcases*}\end{equation*}

We illustrate the construction of $\tilde T$ and $\tilde t$ in the following example.
\begin{ex} Let $\lam=(5,1)$ and let $m=3,n=2$.\label{ex:construction} If we take 
\begin{equation*}T=\young(11122,2)\quad\text{and}\quad t=\young(12345,6),\end{equation*} then 
\begin{equation*}\tilde T=\young(11122,333,2)\quad\text{and}\quad \tilde t=\young(12345,789,6).\end{equation*}
\end{ex}

With these choices of tableaux, we are able to rewrite the signed column sum $\kappa_{\tilde t}$. Taking explicit coset representatives $y_i^{(j)}$ of $C_t^{(j)}$ in $C_{\tilde t}^{(j)}$, where, for $1\leq j\leq m$, $y_i^{(j)}$ is defined to be the transposition $y_i^{(j)}:=\left(\tilde t_i^{(j)}\,\,\,mn+j\right)$ if $i\neq q+1$ and the identity permutation if $i=q+1$, we have that $\kappa_{\tilde t}=\kappa_t\prod_{j=1}^m\left(\sum_{i=1}^{\ell_j+1}\sgn\left(y_i^{(j)}\right)y_i^{(j)}\right)$. For convenience later, we define 
\begin{equation}\label{y_set}Y:=\left\{\left.\prod\nolimits_{j=1}^m y_{x_j}^{(j)}\;\right\vert\; x_j\in\{1,2,\ldots,\ell_j+1\} \text{ for all } 1\leq j\leq m\right\}.
\end{equation}

\subsection{Tableaux for Theorem~\ref{row_am_thm}}\label{am_tableaux}

We also require a candidate $\tilde T$ for the $\tilde\lam$-tableau of type $(m^{n+a})$ needed for Theorem~\ref{row_am_thm}. Given a $\lam$-tableau $T$ of type $(m^n)$, define a $\tilde\lam$-tableau $\tilde T$ of type $(m^{n+a})$ in the following way: 
\begin{equation*}\tilde T_i^{(j)}:=\begin{dcases*}\gamma &if $i=1$ and $j\in\{(\gamma-1)m+1,(\gamma-1)m+2,\ldots,\gamma m\}$ with $1\leq \gamma\leq a$;\\ T_{i-1}^{(j)}+a &if $i>1$. \end{dcases*}
\end{equation*} 
If $T$ is semistandard, then the construction of $\tilde T$ ensures that $\tilde T$ is also semistandard. Again, take $t$ to be the $\lam$-tableau which has the digits $1,2,\ldots,mn$ in increasing order along rows. Define $\tilde t$ to be the labelling tableau \begin{equation*}\tilde t_i^{(j)}:=\begin{dcases*} mn+j &if $i=1$ and $j\in\{1,2,\ldots,am\}$;\\ t_{i-1}^{(j)}&if $i>1$.\end{dcases*}\end{equation*}

\begin{ex} Let $\lam=(3,1)$, $m=2$, $n=2$ and let $T$, $t$ be as in Example~\ref{sstd_hom_ex}. If $a=3$, then \begin{equation*}\tilde T=\young(112233,445,5)\quad\text{and}\quad \tilde t=\young(56789\dc,123,4).\end{equation*}\smallskip
\end{ex}

Just as in \S\ref{m_tableaux}, we write $\kappa_{\tilde t}=\kappa_t\prod_{j=1}^{\lambda_1}\left(\sum_{i=1}^{\ell_j+1}\sgn\left(v_i^{(j)}\right)v_i^{(j)}\right),$ where, for any $1\leq j\leq \lam_1$, the coset representative $v_i^{(j)}$ of $C_t^{(j)}$ in $C_{\tilde t}^{(j)}$ is defined to be the transposition $v_i^{(j)}:=\left(\tilde t_i^{(j)}\,\,\,mn+j\right)$ if $i\neq 1$ and the identity permutation if $i=1$. For convenience later, we define \begin{equation}\label{y_set2}V:=\left\{\left.\prod\nolimits_{j=1}^{\lambda_1} v_{x_j}^{(j)}\;\right\vert\; x_j\in\{1,2,\ldots,\ell_j+1\} \text{ for all } 1\leq j\leq \lambda_1\right\}.\end{equation}

\section{Proof of Theorem~\ref{row_m_thm}}\label{row_m_thm_proof}
We will prove part 2 of Theorem~\ref{row_m_thm} and then indicate how the proof should be modified in order to prove part 1. We make use of the following lemma.

\begin{lem}\label{basis_lem}
If $S^\lam$ appears in $K^{(m^n)}$ with multiplicity $r\geq 0$, then there exist tableaux \linebreak$T_1,\ldots, T_r\in \mathcal{T}_0\big(\lam,(m^n)\big)$ such that \begin{equation*}\big\{\ubar{\theta}_{T_1},\ldots, \ubar{\theta}_{T_r}\;\big\vert\;\ubar{\theta}_{T_i}:S^\lam\to K^{(m^n)}\;\forall\; 1\leq i\leq r\big\}\end{equation*} is a basis for $\Hom_{\mathbb{C}\mathfrak{S}_{mn}}\big(S^\lam,K^{(m^n)}\big)$.
\end{lem} 

\begin{proof} Since $\left\{\theta_T\mid T\in\mathcal{T}_0\big(\lam,(m^n)\big)\right\}$ is a basis for $\Hom_{\mathbb{C}\mathfrak{S}_{mn}}\big(S^\lam,M^{(m^n)}\big)$, $\left\{\ubar\theta_T\mid T\in\mathcal{T}_0\big(\lam,(m^n)\big)\right\}$ spans $\Hom_{\mathbb{C}\mathfrak{S}_{mn}}\big(S^\lam,K^{(m^n)}\big)$. Pruning the spanning set yields a basis. Moreover, since we assumed that $\dim\Hom_{\mathbb{C}\mathfrak{S}_{mn}}\big(S^\lam,K^{(m^n)}\big)=r$, the basis elements will be labelled by tableaux $T_1,\ldots,T_r\in\mathcal{T}_0\big(\lam,(m^n)\big)$.\end{proof}

The next lemma is sufficient to prove existence of $S^{\tilde\lam}$ as a composition factor in $K^{(m^{n+1})}$. Recall that we define $\tilde t$ and $\tilde T$ as in \S\ref{m_tableaux}.

\begin{lem}\label{existence_lem} Under the assumptions of part 2 of Theorem~\ref{row_m_thm}, if $\left(\ubar\theta_T:S^\lam\to K^{(m^n)}\right)\neq 0$ for some tableau $T\in \mathcal{T}_0\big(\lam,(m^n)\big)$, then $\big(\ubar\theta_{\tilde T}:S^{\tilde \lam}\to K^{(m^{n+1})}\big)\neq 0$.
\end{lem}

\begin{proof} Assume that $\left(\ubar\theta_T:S^{\lam} \rightarrow K^{(m^n)}\right)\neq 0$. Since $S^\lam$ is a cyclic module with generator $e_t$, it follows that $(e_t)\ubar\theta_T\neq 0$. Pick any basis element $\ubar{\mathbf R}$ appearing in $(e_t)\ubar\theta_T$ with non-zero coefficient. Since $\phi$ is surjective, there exists ${\mathbf R}\in \mathcal{T}\big(\lam,(m^n)\big)$ such that $\phi:{\mathbf R}\mapsto \ubar{\mathbf R}$. We may write 
\begin{equation*} (e_t)\ubar\theta_T=\sum_{T'\roweq T}\ubar{T'\kappa_t}= \sum_{\substack{T'\roweq T,\\ \pi\in C_t}}\sgn(\pi)\ubar{T'\pi}
\end{equation*} and the coefficient $\mathscr{C}$ of $\ubar{\mathbf R}$ in $(e_t)\ubar\theta_T$ is then
\begin{equation}\mathscr{C}=\sum_{\substack{T'\roweq T,\\ \pi\in C_t,\,\sigma\in \mathfrak{S}_n:\\ T'\pi=\mathbf{R}*\sigma}}\sgn(\pi)\sgn(\sigma)\neq0.\label{coeff1}
\end{equation}
In the same way, we may obtain an expression for the coefficient $\mathcal{C}$ of $\ubar{\tilde{\mathbf R}}$ in $(e_{\tilde t})\ubar\theta_{\tilde T}$: we find that 
\begin{equation*}\mathcal{C}=\sum_{\substack{T''\roweq  \tilde T,\\ \rho\in C_{\tilde t},\,\tau\in\mathfrak{S}_{n+1}:\\ T''\rho\,=\,\tilde{\mathbf{R}}*\tau}}
\sgn(\rho)\sgn(\tau).
\end{equation*} To prove the lemma, it will suffice to prove that $\mathcal{C}$ is non-zero.

We make an observation which allows us to write $\mathcal{C}$ in a more helpful form: that $T''\roweq\tilde{T}$ if and only if $T'\roweq T$, where $T'\in\mathcal{T}\big(\lam,(m^n)\big)$ is such that $T''=\tilde{T'\phantom.}$. To see this, observe that if $T''\roweq\tilde{T}$, then it is possible to remove row $q+1$ of $T''$ -- the row of length $m$ containing only $(n+1)$s -- leaving a $\lambda$-tableau, say $T'$, which is row equivalent to $T$. The reverse implication is clear.

Using this observation, together with the definition of $Y$ in Equation~\eqref{y_set} and the expression of $\rho\in C_{\tilde t}$ as $\rho=\pi y$ (where $\pi\in C_t$ and $y\in Y$), we have that 
\begin{equation*}\mathcal{C}=\sum_{\substack{T'\roweq T,\\ \pi\in C_t,\,y\in Y,\\ \tau\in \mathfrak{S}_{n+1}:\\ \tilde{T'\phantom.}\!\pi y=\tilde{\mathbf{R}}*\tau}}\sgn(\pi)\sgn(y)\sgn(\tau).
\end{equation*} 
Take $T'\roweq T$, $\pi \in C_t$, $y\in Y$ and $\tau\in\mathfrak{S}_{n+1}$ such that $\tilde{T'\phantom.}\!\pi y=\tilde{\mathbf{R}}*\tau$. Since $\pi\in C_t$, it must fix row $q+1$ of $\tilde{T'\phantom.}$. Thus, $\tilde{T'\phantom.}\!\pi=\tilde{T'\pi}$ and so 
\begin{equation}\label{tilde_R_tau}\tilde{T'\phantom.}\!\pi y=\tilde{\mathbf{R}}*\tau\iff \tilde{T'\pi} y=\tilde{\mathbf{R}}*\tau.
\end{equation}

The construction of $\tilde{\mathbf{R}}$ guarantees that the entries in row $q+1$ of $\tilde{\mathbf{R}}$ are all the same. Since $\tau$ is a relabelling permutation, it follows from the statement in Equation~\eqref{tilde_R_tau} that the entries in row $q+1$ of $\tilde{T'\pi} y$ are all the same. Using a construction argument again, the entries in row $q+1$ of $\tilde{T'\pi}$ are all identical. So, for the entries in row $q+1$ of $\tilde{T'\pi} y$ to also be identical, it must be that $y\in Y$ either fixes row $q+1$ of $\tilde{T'\pi}$ -- in which case $y\in Y$ is the identity permutation, which we denote by $\id$ -- or it must swap every identical entry, which is $n+1$, in row $q+1$ with some digit $\beta\in\mathrm{B}$, where \begin{equation*}\mathrm{B}:=\left\{\beta\in\{1,\ldots,n\}\;\Big\vert\;\beta\text{ appears in precisely the columns }1,\ldots,m\text{ of }\tilde{T'\pi}\right\}.\end{equation*} In the latter case, if, for all $1\leq j\leq m$, $\beta$ appears in row $b_j\neq q+1$ in column $j$ of $\tilde{T'\pi}$, then $y=y_\beta:=\prod_{j=1}^m y_{b_j}^{(j)}$. Define \begin{equation*}Y_0\Big(\tilde{T'\pi}\Big):=\{y\in Y\mid y=\id\text{ or }y=y_\beta \text{ for any }\beta\in\mathrm{B}\}.\end{equation*}

We have just seen that $\tilde{T'\pi}y=\tilde{\mathbf{R}}*\tau$ implies that $y\in Y_0\Big(\tilde{T'\pi}\Big)$, and it is easy to see that if $y\in Y_0\Big(\tilde{T'\pi}\Big)$ then $\tilde{T'\pi}y=\tilde{\mathbf{R}}*\tau$. So, we need only sum over $y\in Y_0\Big(\tilde{T'\pi}\Big)$ and therefore 
\begin{equation*}\mathcal{C}=\sum_{\substack{T'\roweq T,\\ \pi\in C_t,\,y\in Y_0\big(\tilde{T'\pi}\big),\\ \tau\in \mathfrak{S}_{n+1}:\\ \tilde{T'\pi} y=\tilde{\mathbf{R}}*\tau}}\sgn(\pi)\sgn(y)\sgn(\tau).
\end{equation*} 
Moreover, if $y=\id$, then $\tilde{T'\pi}=\tilde{T'\pi}y=\tilde{\mathbf{R}}*\tau$.  If $y=y_\beta$, which swaps every $n+1$ in row $q+1$ with some $1\leq\beta\leq n$, then $y$ has the effect of relabelling $\tilde{T'\pi}$ by the transposition $((n+1)\;\,\beta)\in\mathfrak{S}_{n+1}$ and in this case, $\tilde{T'\pi} y=\tilde{\mathbf{R}}*\tau$ if and only if $\tilde{T'\pi}*((n+1)\;\,\beta)=\tilde{\mathbf{R}}*\tau$.

 Using the fact that $m$ is odd, if $y\in Y_0\Big(\tilde{T'\pi}\Big)$ then $\sgn(y)=1$ if $y=\id$ and $\sgn(y)=(-1)^m=-1$ otherwise. So, we may write $\mathcal{C}$ as \begin{equation}\mathcal{C}=\sum_{\substack{T'\roweq T,\\ \pi\in C_t,\,\tau\in\mathfrak{S}_{n+1}:\\
\tilde{T'\pi}=\tilde{\mathbf{R}}*\tau}}\sgn(\pi)\sgn(\tau)\quad- \sum_{\substack{T'\roweq T,\;\pi\in C_t,\\ \tau\in\mathfrak{S}_{n+1},\, \beta\in{\mathrm B}:\\ \tilde{T'\pi}*\left((n+1)\;\,\beta\right)=\tilde{\mathbf{R}}*\tau}}\sgn(\pi)\sgn(\tau).\label{c_express}
\end{equation}

The requirement that $\tilde{T'\pi}*\left((n+1)\;\,\beta\right)=\tilde{\mathbf{R}}*\tau$ says that $\tilde{T'\pi}$ is a relabelling of $\tilde{\mathbf{R}}$. Let $d$ be the number of digits in the set $\{1,2,\ldots,n\}$ that appear in precisely columns $1,2,\ldots,m$ of $\mathbf{R}$. By construction of $\tilde{\mathbf{R}}$, there are $d+1$ of the digits $\{1,2,\ldots,n+1\}$ in precisely columns $1,2,\ldots,m$ of $\tilde{\mathbf{R}}$. So, since $\tilde{T'\pi}$ is a relabelling of $\tilde{\mathbf{R}}$, this forces $\vert{\mathrm B}\vert=d$. 

Now, consider the first sum in the right hand side of Equation~\eqref{c_express} and observe that $\tilde{T'\pi}=\tilde{\mathbf{R}}*\tau$ implies that $T'\pi$ is a relabelling of $\mathbf{R}$. Indeed, if $\tilde{T'\pi}=\tilde{\mathbf{R}}*\tau$, then $\tau$ must not affect row $q+1$ of ${\mathbf R}$, otherwise $\tilde{T'\pi}$ will not have $(n+1)$s in row $q+1$ (which it must do, by the definition of the $\sim$ construction). So, there exists a unique $\sigma\in\mathfrak{S}_n$ which satisfies $T'\pi=\mathbf{R}*\sigma$: take $\sigma=\tau$, from which it follows that $\sgn(\sigma)=\sgn(\tau)$.

Similarly, considering the second sum in~\eqref{c_express}, we see that $\tilde{T'\pi}*\left((n+1)\;\,\beta\right)=\tilde{\mathbf{R}}*\tau$ implies that $T'\pi$ is a relabelling of $\mathbf{R}$. In this case, define $\sigma:=\tau((n+1)\;\,\beta)$. Note that $\sigma$ fixes $n+1$ and so $\sigma\in\mathfrak{S}_n$. Further, $T'\pi={\mathbf R}*\sigma$ and $\sgn(\sigma)=-\sgn(\tau)$. Thus,
\begin{align*}\mathcal{C}&=\sum_{\substack{T'\roweq T,\\ \pi\in C_t,\,\sigma\in\mathfrak{S}_n:\\ T'\pi=\mathbf{R}*\sigma}}\sgn(\pi)\sgn(\sigma)-\vert{\mathrm B}\vert \sum_{\substack{T'\roweq T,\\ \pi\in C_t,\,\sigma\in\mathfrak{S}_n:\\ T'\pi=\mathbf{R}*\sigma}}\sgn(\pi)\big(-\sgn(\sigma)\big)\\
&=(d+1)\sum_{\substack{T'\roweq T,\\ \pi\in C_t,\,\sigma\in\mathfrak{S}_n:\\ T'\pi=\mathbf{R}*\sigma}}\sgn(\pi)\sgn(\sigma)
\end{align*}
and so, using the expression for $\mathscr{C}$ given in Equation~\eqref{coeff1}, we are finally able to conclude that the coefficient of $\ubar{\tilde{\mathbf R}}$ in $(e_{\tilde t})\ubar\theta_{\tilde T}$ is a non-zero multiple of the coefficient $\mathscr{C}$ of $\ubar{\mathbf R}$ in $(e_t)\ubar\theta_T$: more precisely
\begin{equation}\mathcal{C}=(d+1)\mathscr{C}.\qedhere\label{coeff_rel}
\end{equation}
\end{proof}

To complete the proof of part 2 of Theorem~\ref{row_m_thm}, it remains to prove that the multiplicity with which $S^{\tilde\lam}$ appears as a composition factor in the decomposition of $K^{(m^{n+1})}$ is bounded below by the multiplicity of $S^\lam$ in the decomposition of $K^{(m^n)}$. 

Let $B\subseteq\left\{\ubar{\mathbf R}\;\left\vert\;\mathbf{R}\in\mathcal{T}\big(\lam,(m^n)\big)\right.\right\}$ be a basis for $K^{(m^n)}$. There is a bijection $B\to\mathscr{B}:=\left\{\left.\ubar{\tilde{\mathbf R}}\;\right\vert\;\ubar{\mathbf R}\in B\right\}$ defined on oriented column tabloids by 
\begin{equation*} \ubar{\mathbf R}=\begin{vmatrix}X_1\\ X_2\\ \vdots\\ X_n\end{vmatrix}\longmapsto \begin{vmatrix}X_1\\ X_2\\ \vdots\\ X_n\\ \{mn+1,\ldots, mn+m\}\end{vmatrix}=\ubar{\tilde{\mathbf R}}.
\end{equation*}
This is a direct consequence of the construction of $\tilde{\mathbf R}$ and the definition of the labelling tableau $\tilde t$. So, since $B$ is a basis, and therefore all its elements are distinct, the oriented column tabloids which are elements of $\mathscr{B}=\left\{\left.\ubar{\tilde{\mathbf R}}_j\;\right\vert\; 1\leq j\leq |B|\right\}$ must also be distinct. It follows that the formal sum $\sum_{j=1}^{|B|}\beta_j\ubar{\tilde{\mathbf R}}_j$ is equal to zero only if $\beta_j=0$ for all $1\leq j\leq |B|$. In other words, $\mathscr{B}$ is a linearly independent set, which can be extended to a basis for $K^{(m^{n+1})}$.

Suppose that $S^\lam$ appears in $K^{(m^n)}$ with multiplicity $r\geq 0$. By Lemma~\ref{basis_lem}, there is a basis $\left\{\ubar\theta_{T_1},\ldots,\ubar\theta_{T_r}\right\}$ for $\Hom_{\mathbb{C}\mathfrak{S}_{mn}}\big(S^\lam,K^{(m^n)}\big)$, where $T_1,\ldots,T_r\in\mathcal{T}_0\big(\lam,(m^n)\big)$.
 
For a contradiction, assume that $\sum_{i=1}^r\alpha_i\ubar\theta_{\tilde T_i}=0$ for some scalars $\alpha_i$, which are not all zero. It follows that $(e_{\tilde t})\left(\sum_{i=1}^r\alpha_i\ubar\theta_{\tilde T_i}\right)=0$ and so the coefficient of any basis element of the form $\ubar{\tilde{\mathbf R}}$ in $(e_{\tilde t})\left(\sum_{i=1}^r\alpha_i\ubar\theta_{\tilde T_i}\right)$ is zero. If we let $\mathcal{C}_i$ denote the coefficient of $\ubar{\tilde{\mathbf R}}$ in $(e_{\tilde t})\ubar\theta_{\tilde T_i}$, then $\sum_{i=1}^r\alpha_i\mathcal{C}_i=0$. Applying the result in Equation~\eqref{coeff_rel}, $\mathcal{C}_i=(d+1)\mathscr{C}_i$, where $d$ is the number of digits in the set $\{1,2,\ldots,n\}$ which appear in precisely columns $1,2,\ldots,m$ of $\mathbf{R}$ and so does not depend on $i$. Thus, the coefficient of $\ubar{\mathbf R}$ in $(e_t)\left(\sum_{i=1}^r\alpha_i\ubar\theta_{T_i}\right)$ is $\frac{1}{d+1}\sum_{i=1}^r\alpha_i\mathcal{C}_i$ and so is also zero. We chose $\ubar{\tilde{\mathbf R}}$ arbitrarily and so it follows that $\sum_{i=1}^r\alpha_i\ubar\theta_{T_i}$ maps the generator $e_t$ to zero. This implies that $\sum_{i=1}^r\alpha_i\ubar\theta_{T_i}=0$.  However, since $\big\{\ubar\theta_{T_1},\ldots,\ubar\theta_{T_r}\big\}$ is a linearly independent set, it follows that $\alpha_i=0$ for all $1\leq i\leq r$, but this contradicts the assumptions on $\{\alpha_i\mid 1\leq i\leq r\}$.\medskip

We conclude this section by indicating the modifications that need to be made to the above proofs in order to obtain the proof of part 1 of Theorem~\ref{row_m_thm}. Firstly, we remark that we obtain a statement analogous to Lemma~\ref{basis_lem}, establishing a basis for $\Hom_{\mathbb{C}\mathfrak{S}_{mn}}\left(S^\lam,H^{(m^n)}\right)$ consisting of maps labelled by semistandard tableaux.

In~\S\ref{signed_setting}, we alluded to the fact that in the Foulkes setting we need only consider the pattern of tableaux entries. Therefore, we may disregard any signs corresponding to relabelling permutations $\tau$ and $\sigma$ that appear in the above proofs. Instead of writing, say $T'\pi=\mathbf R*\sigma$, we write $\bar{T'\pi}=\bar{\mathbf R}$ to reflect the fact that $T'\pi$ is a relabelling of $\mathbf R$. However, where it occurs in the proof of Lemma~\ref{existence_lem}, we retain the specification of the relabelling of $\tilde{T'\pi}$ by the transposition $((n+1)\;\,\beta)$ since this is important for the proof. However, we still suppress the action of $\tau$, so that $\tilde{T'\pi}*((n+1)\;\,\beta)=\tilde{\mathbf{R}}*\tau$ is replaced with the statement $\bar{\tilde{T'\pi}*((n+1)\;\,\beta)}=\bar{\tilde{\mathbf{R}}}$, which says that $\tilde{T'\pi}$ is a relabelling of $\mathbf{R}$. It is then sufficient to observe that $\bar{\tilde{T'\pi}}=\bar{\tilde{\mathbf R}}$ implies that $\bar{T'\pi}=\bar{\mathbf R}$; we need not worry about the specific relabelling permutations. The conclusion of the proof of part 1 of Theorem~\ref{row_m_thm} is analogous to that of part 2: we find that the coefficient $\mathcal{C}$ of $\bar{\tilde{\mathbf R}}$ in $(e_{\tilde t})\bar\theta_{\tilde T}$ is a non-zero multiple of the coefficient $\mathscr{C}$ of $\bar{\mathbf R}$ in $(e_t)\bar\theta_T$: more precisely $\mathcal{C}=(d+1)\mathscr{C}$.

\section{Proof of Theorem~\ref{row_am_thm}}\label{row_am_thm_proof}

This proof proceeds in a similar manner to that of Theorem~\ref{row_m_thm}. In the setting of Theorem~\ref{row_am_thm}, we begin by establishing the existence of $S^{\tilde\lam}$ as a composition factor of $H^{(m^{n+a})}$ for any $a\geq 2$. Recall that, for this section, we define $\tilde\lam:=(am,\lam_1,\ldots,\lam_\ell)$ and we redefine $\tilde t$ and $\tilde T$ as in \S\ref{am_tableaux}.

\begin{lem}\label{existence_lem2}
Under the assumptions of Theorem~\ref{row_am_thm}, if $\left(\bar\theta_T:S^\lam\to H^{(m^n)}\right)\neq 0$ for some tableau $T\in \mathcal{T}_0\big(\lam,(m^n)\big)$, then $\big(\bar\theta_{\tilde T}:S^{\tilde \lam}\to H^{(m^{n+a})}\big)\neq 0$.
\end{lem}

\begin{proof}
Assume that $\left(\bar\theta_T:S^{\lam} \rightarrow H^{(m^n)}\right)\neq 0$; it follows that $(e_t)\bar\theta_T\neq 0$. Pick any basis element $\bar{\mathbf R}$ appearing with non-zero coefficient $\mathscr{C}$ in $(e_t)\bar\theta_T$. Since $\psi$ is surjective, there exists ${\mathbf R}\in\mathcal{T}\big(\lam,(m^n)\big)$ such that $\psi:{\mathbf R}\mapsto{\bar{\mathbf R}}$.  An expression for the coefficient $\mathscr{C}$ is 
\begin{equation}\mathscr{C}=\sum_{\substack{T'\roweq T,\\ \pi\in C_t:\\ \bar{T'\pi}=\bar{\mathbf R}}}\sgn(\pi).\label{coeff3}
\end{equation}Fix $a\geq 2$. It will suffice to show that the coefficient $\mathcal{C}$ of $\bar{\tilde{\mathbf R}}$ in $(e_{\tilde t})\bar\theta_{\tilde T}$ is non-zero, where 
\begin{equation*}\mathcal{C}:=\sum_{\substack{T''\roweq  \tilde T,\\ \rho\in C_{\tilde t}:\\ \bar{T''\!\rho}=\bar{\tilde{\mathbf R}}}}\sgn(\rho).
\end{equation*}

Firstly, recall from \S\ref{am_tableaux} that $\rho$ may be expressed as $\rho=\pi v$ for some unique $\pi\in C_t$ and $v\in V$, the definition of $V$ being that given in Equation~\eqref{y_set2}. Secondly, take $T''\roweq \tilde T,\,\pi\in C_t$ and $v\in V$ such that $\bar{T''\pi v}=\bar{\tilde{\mathbf R}}$. Entries in row 1 of $T''\pi v$ must have the same pattern as entries in row 1 of $\tilde{\mathbf R}$, the latter being $1\ldots 1\,2\ldots 2\ldots a\ldots a$, with $m$ copies of each digit. Also note that the first row of $\tilde{\mathbf R}$ is the same as the first row of $\tilde T$. Since there is only one entry in columns $\lam_1+1,\ldots,am$ of $T''$, $\pi v$ fixes these columns. Hence, entries in columns $\lambda_1+1,\ldots,am$ of $T''$ must be a relabelling (by $\omega\in\mathfrak{S}_a$, say) of the entries in columns $\lambda_1+1,\ldots,am$ of $\tilde{\mathbf R}$. Since $\lam_1<2m$, to preserve the pattern of the first row, we must have the entries in columns $m+1,\ldots,\lam_1$ in row 1 of $T''$ equal to the entry in columns $\lam_1+1,\ldots,2m$, which is $(2)\omega$. The fact that $T''\roweq \tilde T$ tells us that there is one remaining digit (repeated $m$ times) which is the entry in columns $1,\ldots,m$ of row 1 of $T''$. We conclude that row 1 of $T''$ is of the form \begin{equation*}\underbracket[0.5pt]{(1)\omega\,\,(1)\omega\,\,\cdots\,\,(1)\omega}_{m\text{ copies}}\,\,\underbracket[0.5pt]{(2)\omega\,\,(2)\omega\,\,\cdots\,\,(2)\omega}_{m\text{ copies}}\,\,\cdots\cdots\cdots\,\,\underbracket[0.5pt]{(a)\omega\,\,(a)\omega\,\,\cdots\,\,(a)\omega}_{m\text{ copies}},
\end{equation*}where $\omega \in\mathfrak{S}_a$. Entries in the remaining rows of $T''$ are ${T''}_{i+1}^{(j)}:={T'}_{i}^{(j)}+a$ (where $1\leq i\leq \ell$) for some $T'\roweq T$, that is, rows $2,3,\ldots,\ell+1$ of $\tilde{T'}$. It follows that $T''=\tilde{T'}*\omega$, where $\omega\in \mathfrak{S}_a\subseteq\mathfrak{S}_{a+n}$, and so the expression for $\mathcal{C}$ becomes 
\begin{equation*}\mathcal{C}=\sum_{\substack{T'\roweq T,\,\omega\in \mathfrak{S}_a,\\ \pi\in C_t,\,v\in V:\\ \bar{(\tilde{T'}*\omega)\pi v}=\bar{\tilde{\mathbf R}}}}\sgn(\pi)\sgn(v).
\end{equation*}

We must determine the $v\in V$ for which $\bar{\big(\tilde{T'}*\omega\big)\pi v}=\bar{\tilde{\mathbf R}}$ holds. Recall that $\pi\in C_t$ fixes row $1$ of $T''$. Therefore, $\bar{\big(\tilde{T'}*\omega\big)\pi v}=\bar{\tilde{\mathbf R}}$ if and only if $\bar{\big(\tilde{T'\pi}*\omega\big)v}=\bar{\tilde{\mathbf R}}$. To preserve the pattern of row 1 of $\tilde{T'\pi}*\omega$, $v$ must either fix row 1; or swap every $(1)\omega$ in row 1 with a digit $\beta\in\mathrm{B}$ (and fix every $(2)\omega,\ldots,(a)\omega$ because $\lam_1<2m$), where \begin{equation*}{\mathrm B}:=\left\{\beta\in\{a+1,\ldots,a+n\}\;\Big\vert\; \beta \text{ appears in precisely columns }1,2\ldots,m\text{ of }\tilde{T'\pi}*\omega\right\}.\end{equation*} In other words, if, for all $1\leq j\leq m$, $\beta$ appears in row $b_j\neq 1$ in column $j$ of $\tilde{T'\pi}*\omega$, then \begin{equation*}v\in V_0:=\big\{\id\big\}\cup\left\{v_\beta:=\prod\nolimits_{j=1}^m v_{b_j}^{(j)} \text{ for any }\beta\in{\mathrm B}\right\}\subseteq V;\end{equation*} note that $V_0$ depends on $\tilde{T'\pi}*\omega$. In fact, $\bar{\big(\tilde{T'\pi}*\omega\big)v}=\bar{\big(\tilde{T'\pi}*\omega\big)}$ if and only if $v\in V_0$.  

Let $d$ be the number of digits in the set $\{1,\ldots,n\}$ which appear in precisely columns $1,\ldots,m$ of $\mathbf R$. By construction of $\tilde{\mathbf R}$ and the requirement that $\bar{\tilde{T'\pi}*\omega}=\bar{\tilde{\mathbf{R}}}$, we deduce that $\vert V_0\setminus\{\id\}\vert=d$. Also, since $m$ is even, if $v\in V_0$ then $v$ is even. This knowledge allows us to write $\mathcal{C}$ as
\begin{equation*}\mathcal{C}=\sum_{\substack{T'\roweq T,\\ \omega\in \mathfrak{S}_a,\,\pi\in C_t:\\ \bar{\tilde{T'\pi}*\omega}=\bar{\tilde{\mathbf{R}}}}}\vert V_0\vert\sgn(\pi)=(d+1)\sum_{\substack{T'\roweq T,\\ \omega\in \mathfrak{S}_a,\,\pi\in C_t:\\ \bar{\tilde{T'\pi}*\omega}=\bar{\tilde{\mathbf{R}}}}}\negthickspace\sgn(\pi).
\end{equation*}

Finally, we should observe that $\bar{\tilde{T'\pi}*\omega}=\bar{\tilde{\mathbf{R}}}$ implies that $\bar{T'\pi}=\bar{\mathbf R}$. Conversely, given any $T'\roweq T$ and $\pi\in C_t$ such that $\bar{T'\pi}=\bar{\mathbf R}$, setting $T''=\tilde{T'}*\omega$ for some $\omega\in \mathfrak{S}_a$ we find that $T''\roweq \tilde T$ and for $v\in V_0$, $\bar{T''\pi v}=\bar{\tilde{T'\pi}*\omega}=\bar{\tilde{\mathbf R}}$. Hence, using the expression for $\mathscr{C}$ given in Equation~\eqref{coeff3}, we conclude that the coefficient of $\bar{\tilde{\mathbf R}}$ in $(e_{\tilde t})\bar\theta_{\tilde T}$ is a non-zero multiple of the coefficient $\mathscr{C}$ of $\bar{\mathbf R}$ in $(e_t)\bar\theta_T$: more precisely 
\begin{equation}\mathcal{C}=\vert \mathfrak{S}_a\vert(d+1)\sum_{\substack{T'\roweq T,\\ \pi\in C_t:\\ \bar{T'\pi}=\bar{\mathbf{R}}}}\sgn(\pi)=a!(d+1)\mathscr{C}.\qedhere\label{coeff_rel2}
\end{equation}
\end{proof}

It just remains to verify the bound on the multiplicity with which $S^{\tilde\lam}$ appears as a composition factor in the decomposition of $H^{(m^{n+a})}$. We will use the fact that, if $B\subseteq\left\{\bar{\mathbf R}\;\left\vert\;{\mathbf R}\in \mathcal{T}\big(\lam,(m^n)\big)\right.\right\}$ is a basis for $H^{(m^n)}$, then the set $\mathscr{B}:=\left\{\left.\bar{\tilde{\mathbf R}}\;\right\vert\;\bar{\mathbf R}\in B\right\}\subseteq H^{(m^{n+a})}$ is linearly independent.  Indeed, for any $a\geq 2$, there exists a bijection $B\to\mathscr{B}$ defined on set partitions by 
\begin{equation*}
\{X_1,\ldots,X_n\}\longmapsto\big\{X_1,\ldots,X_n,\{mn+1,\ldots,mn+m\},\ldots,\{mn+(a-1)m+1,\ldots,mn+am\}\big\}.
\end{equation*}
Since elements of $B$ are distinct, elements of $\mathscr{B}$ must also be distinct. Thus, the formal sum $\sum_{j=1}^{|B|}\beta_j\bar{\tilde{\mathbf R}}_j$ is equal to zero only if $\beta_j=0$ for all $1\leq j\leq |B|$. 

Suppose that $S^\lam$ appears in $H^{(m^n)}$ with multiplicity $r\geq 0$. There is a basis $\left\{\bar\theta_{T_1},\ldots,\bar\theta_{T_r}\right\}$ for $\Hom_{\mathbb{C}\mathfrak{S}_{mn}}\!\big(S^\lam,H^{(m^n)}\big)$, where $T_1,\ldots,T_r\in\mathcal{T}_0\big(\lam,(m^n)\big)$. For a contradiction, assume that $\sum_{i=1}^r\alpha_i\bar\theta_{\tilde T_i}=0$ for some scalars $\alpha_i$, which are not all zero. It follows that $(e_{\tilde t})\left(\sum_{i=1}^r\alpha_i\bar\theta_{\tilde T_i}\right)=0$. However, for any $\mathbf R\in\mathcal{T}\big(\lam,(m^n)\big)$, the coefficient of a basis element $\bar{\tilde{\mathbf R}}$ in $(e_{\tilde t})\left(\sum_{i=1}^r\alpha_i\bar\theta_{\tilde T_i}\right)$ is, by the result in Equation~\eqref{coeff_rel2}, a non-zero multiple of the coefficient of $\bar{\mathbf R}$ in $(e_t)\left(\sum_{i=1}^r\alpha_i\ubar\theta_{T_i}\right)$. It follows that $(e_t)\left(\sum_{i=1}^r\alpha_i\ubar\theta_{T_i}\right)=0$ and thus $\alpha_1=\ldots=\alpha_r=0$, which contradicts the assumptions on $\{\alpha_i\mid 1\leq i\leq r\}$. 

\section{Conjectures}\label{conjectures}
We conjecture that two of the results stated in Section~\ref{existing} will generalise.  In particular, our first two conjectures are generalisations of Foulkes' Second Conjecture (\ref{FSC}) and Newell's result (\ref{Newell_statement}), respectively.
\begin{conj}If $\lam\vdash mn$, then $p^\lam_{\nu,(m)}\leq p^{\lam+(n)}_{\nu,(m+1)}$ for any $\nu\vdash n$.\label{conj1}
\end{conj}

\begin{rem} The proof of~\ref{FSC} is not the only evidence in support of Conjecture~\ref{conj1}.  The conjecture has also been proved in the case $\nu=(1^n)$ in~\cite[Theorem~4.3.4]{deboeck}.\end{rem}

\begin{conj}If $\lam\vdash mn$, then $p^\lam_{\nu,(m)}=p^{\lam+(1^n)}_{\nu',(m+1)}$ for any $\nu\vdash n$.\label{conj2}
\end{conj}

The following conjecture is a consequence of Conjecture~\ref{conj2}.
\begin{conj}If $\lam\vdash mn$, then $p^\lam_{\nu,(m)}=p^{\lam+(2^n)}_{\nu,(m+2)}$ for any $\nu\vdash n$.\label{conj3}
\end{conj}

\section*{Acknowledgements}
The author has been supported by the Engineering and Physical Sciences Research Council (grant number EP/P505577/1), and the School of Mathematics, Statistics and Actuarial Science, University of Kent.

\bibliographystyle{plain}
\bibliography{mjdeb_sstd_homom_paper_v4}

\end{document}